\documentclass[11pt]{article}
\usepackage{fancyhdr}
\usepackage{graphicx}
\usepackage{amsmath,amsthm,amsfonts,amssymb,amscd}
\usepackage{commath}
\usepackage[utf8]{inputenc}
\usepackage[english]{babel}
\usepackage[headsep=0.5cm,headheight=2cm]{geometry}
\usepackage[shortlabels]{enumitem}
\usepackage{hyperref}
\usepackage{mathtools}
\usepackage{etoolbox}
\usepackage{tikz-cd}
\usepackage{xcolor}

\hypersetup{
	colorlinks=true,
	linkcolor=blue,
	citecolor= magenta,
	filecolor=magenta,      
	urlcolor=cyan,
	pdftitle={LUR fully closed},
	pdfpagemode=FullScreen,
}
\newcommand{\R}{\mathbb{R}}
\newcommand{\N}{\mathbb{N}}

\newcommand{\Q}{\mathbb{Q}}

\DeclareMathOperator{\osc}{osc}
\DeclareMathOperator{\diam}{diam}

\DeclareMathOperator{\dist}{dist}
\DeclareMathOperator{\lgth}{lh}
\DeclareMathOperator{\sht}{sh}

\newtheorem{theorem}{Theorem}[section]
\newtheorem{lemma}[theorem]{Lemma}
\newtheorem{corollary}[theorem]{Corollary}

\newtheorem{proposition}[theorem]{Proposition}
\newtheorem{claim}[theorem]{Claim}

\theoremstyle{definition}
\newtheorem{definition}[theorem]{Definition}
\newtheorem{notation}[theorem]{Notation}

\theoremstyle{remark}
\newtheorem*{remark}{Remark}

\allowdisplaybreaks

\title{Continuous functions on limits of F-decomposable systems}
\author{Todor Manev}
\date{}

\begin{document}

\maketitle

\begin{abstract}
	We introduce the concept of F-decomposable systems, well-ordered inverse systems of Hausdorff compacta with fully closed bonding mappings. A continuous mapping between Hausdorff compacta is called fully closed if the intersection of the images of any two closed disjoint subsets is finite. We give a characterization of such systems in terms of a property of the continuous functions on their limit. When, moreover, the fibers of neighboring bonding mappings are metrizable, we call the limit of such a system an $\mathrm{F}_d$-compact, a particular case of a Fedorchuk compact. The stated property allows us to obtain a locally uniformly rotund renorming on the space $C(K)$, where $K$ is an $\mathrm{F}_d$-compact of countable spectral height.
\end{abstract}

\setlength{\abovedisplayskip}{4pt}
\setlength{\belowdisplayskip}{4pt}

\section{Introduction}

In \cite{fed69} V. V. Fedorchuk introduced a technique for constructing topological spaces that proved effective in giving counterexamples to some statements in dimension theory. It has since been widely used in all branches of topology. The technique is referred to in the literature as a resolution. It consists of replacing points in some topological space by closed sets, using mappings on the original space to describe the resulting topology. In \cite{watson} Watson argues that the technique is natural and that it has been used in some form by every topologist, often unknowingly.

In the same paper of 1969 Fedorchuk defined a property of mappings between topological spaces that he called \emph{fully closed}. The canonical mapping from the resolution of a space onto the original space always satisfies this property. In fact, with some additional assumptions, namely if the fibers of a mapping $f: X \to Y$ are absolute neighborhood retracts, its full closedness is equivalent to $X$ being obtainable from $Y$ by a resolution. We shall now state the precise definition, but before that we need to fix some notation. For a mapping $f:M \to N$ between sets and a subset $A \subset M$, we will denote with $f^\#(A)$ the \textit{small image} of $A$, that is, $f^\#(A) := N \setminus f(M \setminus A)$.

\begin{definition} \label{fully_closed}
	Let $X$ and $Y$ be topological spaces and $f: X \to Y$ a continuous mapping. $f$ is called \textit{fully closed} at $y \in Y$ if for any finite open cover $\{U_1, ..., U_s\}$ of $f^{-1}(y)$ the set $\{y\} \cup \bigcup_{i = 1}^s f^\#\left(U_i\right)$ is a neighborhood of $y$. A continuous surjective mapping is called \textit{fully closed} if it is fully closed at every point of $Y$.
\end{definition}

In Fedorchuk's constructions the method of resolutions, or more generally the method of fully closed mappings, was used iteratively. More precisely, topological spaces were constructed as limits of inverse systems, where every subsequent level was obtained from the previous one by a resolution, or respectively every successive level mapped onto the previous by a fully closed mapping. This motivated Ivanov to define in \cite{ivanov1984fedorchuk} a class of compacta that he called \emph{Fedorchuk compacta} or \emph{F-compacta}. We give the precise definition bellow along with the definitions of several other important concepts. For more information on inverse systems and their limits we refer the reader to \cite[Chapter 2.5]{eng} or \cite{fedFC}.

\begin{definition} \label{F-system}
	Let $S := \{X_\alpha, \pi^\beta_\alpha, \alpha, \beta \in \gamma\}$ be a well-ordered continuous inverse system of Hausdorff compacta with $X_0$ a singleton. We say that $S$ is an \emph{F-system} if the neighboring bonding mappings, $\pi^{\alpha+1}_\alpha$ for $\alpha + 1 \in \gamma$, are fully closed. We say that $S$ is an \emph{F-decomposable system} if all bonding mappings, $\pi^\alpha_\beta$ for $\alpha, \beta \in \gamma$, are fully closed.
\end{definition}

\begin{definition}\label{F-comp}
	Let $S := \{X_\alpha, \pi^\beta_\alpha, \alpha, \beta \in \gamma\}$ be an F-system for which the fibers of all neighboring bonding mappings are metrizable. Then the limit $X := \lim\limits_{\leftarrow} S$ is called a \emph{Fedorchuk compact} (\emph{F-compact}) and the length of the system is called its spectral height (denoted $\sht(X)$), provided that $X$ cannot be obtained from a system satisfying the above conditions of strictly smaller length. If $S$ is moreover F-decomposable, we say that $X$ is an \emph{$\mathrm{F}_d$-compact}.
\end{definition}

Some classical topological spaces can be obtained from metrizable spaces by resolutions, thus making them F-compacta of spectral height 3. Examples include the double arrow space, the Alexandroff doubling of a metric space, the lexicographic square. More generally, the lexicographic cube $[0,1]^\alpha_{lex}$ is a Fedorchuk compact, as is the lexicographic product of arbitrary many copies of $\{0,1\}$. These examples are in fact $\mathrm{F}_d$-compacta. In \cite{manev24} it was shown that any scattered compact space can be obtained as the limit of an F-decomposable system with the fibers of neighboring bonding mappings homeomorphic to the one-point compactification of a discrete set.

Fedorchuk compacta and more generally F-systems have continued to produce counterexamples, including in Banach space theory (see e.g. \cite{av_kosz}, \cite{kosz-few_op}, \cite{mirna_pleb}). Recent interest in them has emerged in relation to the theory of renormings of Banach spaces. The present work is primarily devoted to such investigations. Recall that the norm of a Banach space $E$ is called \emph{locally uniformly rotund} (\emph{LUR} for short) if for any point $x$ in its unit sphere $S_E$ and any sequence $(x_n)_{n \in \N} \subset S_E$, we have the implication
\begin{equation*}
	\lim_{n \to \infty}\left\|\frac{x + x_n}{2}\right\| = 1 \quad \implies \quad \lim_{n \to \infty}\|x - x_n\| = 0.
\end{equation*}

The question for which compact spaces $K$ there exists an equivalent LUR norm on $C(K)$ has been studied extensively throughout the years. Notable classes for which the result is positive include Eberlein compacta, Valdivia compacta (see e.g. \cite[Section VII.7]{dgz}), compact spaces of functions on a Polish space with countably many discontinuities (\cite{hmo-count.disc}). The latter result gives, in particalur, a LUR renorming on $C(K)$, where $K$ is the Helly compact. In \cite{hay-rog_scattered} a positive result was obtained for any scattered compact of countable height. For the case of $K$ the one-point compactification of a tree, LUR renormability and many other renorming properties of $C(K)$ were characterized by Haydon in \cite{haydon_trees}. In \cite{hjnr00} it was shown that $C([0,1]^\alpha_{lex})$ is LUR renormable if and only if $\alpha$ is a countable ordinal. Since the lexicographic cube is an $\mathrm{F}_d$-compact, the question arises whether this result is extendable to the whole class, or even to all F-compacta of countable spectral height.

The first result in this direction was given in \cite{gist} where it was proven that $C(K)$ admits an equivalent LUR norm whenever $K$ is a Fedorchuk compact of spectral height $3$. This was then extended in \cite{manev24} to F-compacta of finite spectral height. In Section \ref{sect_proof} we prove that an equivalent pointwise-lower semicontinuous LUR norm can be constructed on $C(K)$ whenever $K$ is an $\mathrm{F}_d$-compact with $\sht(K) < \omega_1$. An important step in the proof is a technical lemma (Lemma \ref{finite_branching}) in Section \ref{characterization} that gives us a rigidity condition to obtain the renorming. This lemma also allows us to give a characterization of F-decomposable systems in terms of a relation between continuous functions on their limit and continuous functions on an associated tree.

\section{Preliminaries}

We start this section with a well-known fact that will later be useful for evaluating the distances between some continuous functions.

\begin{proposition}[see e.g. {\cite[p. 173]{holmes_opt}}] \label{metr_proj}
	Let $X$ and $Y$ be Hausdorff compact spaces and $\varphi: X \to Y$ a continuous surjection. Let $\varphi^0$ denote the natural embedding of $C(Y)$ into $C(X)$, that is, $(\varphi^0 g)(x) := g(\varphi(x))$, for $g \in C(Y)$. Then, whenever $f \in C(X)$, there exists a function $g \in C(Y)$ such that $\|f- \varphi^0 g\| = \dist(f, \varphi^0 C(Y)) = \frac{1}{2} \sup \{\osc_{\varphi^{-1}(y)} f: y \in Y\}$.
\end{proposition}

\begin{proof}
	Denote $\alpha := \sup \{\osc_{\varphi^{-1}(y)} f: y \in Y\}$ and consider the functions $\Phi, \Psi: Y \to \R$ given by 
	\begin{align*}
		&\Phi(y) := \max\{f(x): x \in \varphi^{-1}(y)\} - \frac{\alpha}{2};\\
		&\Psi(y) := \min\{f(x): x \in \varphi^{-1}(y)\} + \frac{\alpha}{2}.
	\end{align*}
	Then $\Phi$ is upper semicontinuous, $\Psi$ is lower semicontinuous, and $\Phi(y) \leq \Psi(y)$ for any $y \in Y$. Thus the set-valued mapping $F : Y \rightrightarrows \R$, given by $F(y) := [\Phi(y), \Psi(y)]$, has nonempty, closed, convex values and is lower semicontinuous. By the Michael selection theorem, we find a continuous selection $g$ of $F$. We immediately have $f(x) - g(\varphi(x)) \leq \frac{\alpha}{2}$ for any $x \in X$, as desired.
\end{proof}

\begin{remark}
	Throughout what follows, in order to simplify notation, whenever we have a continuous surjection $\varphi: X \to Y$ between Hausdorff compacta, we will identify $C(Y)$ with its natural embedding, that is, we consider $C(Y)$ as a subspace of $C(X)$.
\end{remark}

The notation that we now introduce will be used extensively throughout the article.

\begin{notation} \label{quot}
	Let $X$ and $Y$ be topological spaces and $f$ a continuous mapping from $X$ onto $Y$. If $M \subset Y$, by $Y_f^M$ we will denote the quotient space corresponding to the following equivalence classes:
	\begin{equation*}
		[x] \quad = \quad \left\{\begin{aligned} \quad
			&x, \qquad &x \in f^{-1}(M);\\
			&f^{-1}\left(f(x)\right), \qquad &f(x) \in Y \setminus M.
		\end{aligned}\right.
	\end{equation*}
	We will denote the corresponding quotient mapping from $X$ to $Y_f^M$ by $p_f^M$, or by $f^M$ when appropriate, and by $\pi_f^M: Y_f^M \to Y$ the unique mapping such that $f = \pi_f^M \circ p_f^M$.
	If $M = \{y\}$, $y \in Y$, we will simply write $Y^y$. We may occasionally add a subscript ($Y_f^M$) if the mapping generating the equivalence relation is not clear from the context.
\end{notation}

\begin{center}
	\begin{tikzcd}
		X \arrow[rr, "f"] \arrow[rd, "p_f^M"] && Y \\
		& Y^M \arrow[ru, "\pi_f^M"]
	\end{tikzcd}
\end{center}

Many equivalent formulations of the property of full closedness have been established throughout the years. Different definitions are more convenient depending on the context. Bellow we state the ones that will be relevant for what follows.

\begin{proposition} [{\cite[Theorem II.1.6]{fedFC}}] \label{eqDef}
	Let $X$ and $Y$ be Hausdorff compact spaces and $f: X \to Y$ a continuous onto mapping. Then the following are equivalent:
	\begin{enumerate}[(1)]
		\item $f$ is fully closed. \label{def1}
		\item If $F_1, F_2 \subset X$ are closed and disjoint, the set $f(F_1) \cap f(F_2)$ is finite. \label{def4}
		\item If $U \subset X$ is open and $y \in Y$, the set $U^y := \left(f^{-1}(y) \cap U\right) \cup f^{-1}\left(f^\#(U)\right)$ is open. \label{def5}
		\item For any $M \subset Y$, the space $Y^M$ is Hausdorff. \label{def6}
	\end{enumerate}
\end{proposition}

From \ref{def4} it is obvious that if the composition of two continuous mappings $g \circ f$ is fully closed, then its left divisor $g$ is also fully closed. In general this is not the case for the right divisor $f$. However, the following proposition gives us a sufficient condition for the full closedness of this mapping at some point. This will be particularly useful in Section \ref{constr}.

\begin{proposition}[{\cite[II.1.9]{fedFC}}] \label{comp_fc}
	Let $f: X \to Y$ and $g: Y \to Z$ be continuous mappings between Hausdorff compact spaces. Suppose $g \circ f$ is fully closed at a point $z \in Z$ and let $y \in g^{-1}(z)$. Then if $f$ or $g$ is one-to-one at $y$, the mapping $f$ is fully closed at $y$.
\end{proposition}

\begin{corollary}[{\cite[II.1.10]{fedFC}}] \label{quot_fc}
	Let $f: X \to Y$ be a fully closed surjection between Hausdorff compact spaces and let $M \subset Y$. The quotient mapping $p_f^M: X \to Y^M_f$ is also fully closed.
\end{corollary}

The proposition that follows gives us metrizability conditions for the domain space of a fully closed mapping. It will be important in Section \ref{sect_proof}.

\begin{proposition} [{\cite[II.3.10]{fedFC}}] \label{Fed_metr}
	Let $f: X \to Y$ be a fully closed mapping between Hausdorff compacta. Then $X$ is metrizable if and only if the following conditions hold:
	\begin{itemize}[noitemsep]
		\item $Y$ is metrizable;
		\item All the fibers $f^{-1}(y)$ are metrizable;
		\item The set of nontrivial fibers $\left\{y \in Y: \left|f^{-1}(y)\right| \geq 2\right\}$ is countable.
	\end{itemize}
\end{proposition}

This next proposition first appeared in \cite{gist} and is key to the results there, in \cite{manev24}, and in the present work. It gives another characterization of full closedness, this time in terms of the behavior of continuous functions over the domain space.

\begin{proposition}[\cite{gist}, see also \cite{manev24}]\label{c_0.osc}
	Let $\pi: X \to Y$ be a continuous surjective mapping between Hausdorff compacta. Then $\pi$ is fully closed if and only if for all $f \in C(X)$ we have that $\left(\osc_{\pi^{-1}(y)}f\right)_{y \in Y} \in c_0(Y)$.
\end{proposition}

We now provide the main result on which our technique for obtaining a LUR renorming of the space of continuous functions on an $\mathrm{F}_d$-compact will be based. This characterization of Molt\`o, Orihuela and Troyanski from 1997 is the foundation of the nonlinear transfer technique developed in \cite{motv}.

\begin{theorem}[\cite{mot97}]\label{LUR_char}
	Let $E$ be a Banach space and $F$ a norming subspace of its dual. Then $E$ admits an equivalent $\sigma(E,F)$-lower semicontinuous LUR norm if and only if for any $\epsilon >0$ there exists a countable decomposition $E = \bigcup_{n \in \N} E_n$ such that for all $n \in \N$ and $x \in E_n$ there exists a $\sigma(E,F)$-open halfspace $H$ containing $x$ and satisfying:
	\begin{equation*}
		\|\cdot\| \text{-} \diam \left(E_n \cap H \right) < \epsilon.
	\end{equation*}
\end{theorem}


\section{F-decomposable systems}

Fedorchuk's original technique, as well as modifications used by Ivanov, among others, consists of building inverse systems of topological spaces where the neighboring bonding mappings are fully closed. However, in most of the classical examples all bonding mappings are fully closed, thus making the systems F-decomposable. For an example of a system where this is not the case we refer the reader to \cite{ivanov_qfc}.

As we shall see, this additional property is related to the structure of the space of continuous functions on the limit of such a system. In fact, in Section \ref{characterization} we characterize F-decomposable systems in terms of a property of oscillations of continuous functions on the limit.

We start this section with two propositions that already give us a sense of the nice properties of inverse systems with fully closed bonding mappings. Proposition \ref{fed_level} will be the stepping stone for the construction in Section \ref{constr}.

\begin{proposition} [{\cite[II.3.1]{fedFC}}] \label{fc_limit}
	Let $X = \lim\limits_{\leftarrow}\{X_\alpha, \pi_\alpha^\beta, \alpha, \beta \in \gamma\}$ be the limit of an F-decomposable system. Then the limit mappings $\pi_\alpha: X \to X_\alpha$ are fully closed.
\end{proposition}

\begin{proposition}[{\cite[II.1.13]{fedFC}}] \label{fed_level}
	Let $f: X \to Y$ and $g: Y \to Z$ be fully closed mappings between Hausdorff compact spaces with a fully closed composition $g \circ f$. Let $M \subset Z$. Then the mapping $p_g^M \circ f: X \to Z^M$ is also fully closed. 
\end{proposition}

\begin{center}
	\begin{tikzcd}
		X \arrow[rr, "f"]  \arrow[rrrd, bend right=20, "p_g^M \circ f "] & &
		Y \arrow[rr, "g"] \arrow[rd, "p_g^M"] & &
		Z \\
		& & & Z^M \arrow[ru, "\pi_g^M"] &
	\end{tikzcd}
\end{center}

\subsection{Construction} \label{constr}

	\begin{figure}[h]
		\centering
		\begin{tikzcd}
			X \arrow[rr, "f_2"]  \arrow[rrrd, bend right=10, "p^{M_1} \circ f_2 "] \arrow[rdd, "q"] & &
			Y \arrow[rr, "f_1"] \arrow[rd, "p^{M_1}"] & &
			Z \\
			& & & Z^{M_1} \arrow[ru, "\pi^{M_1}"] & \\
			& \left(Z^{M_1}\right)^{p^{M_1}(M_2)} \arrow[rru, "\pi^{p^{M_1}(M_2)}"]	&
		\end{tikzcd}
		\caption{A diagram of a three-level construction. Here $M_1 \subset Z$ and $M_2 \subset Y$.}
	\end{figure}
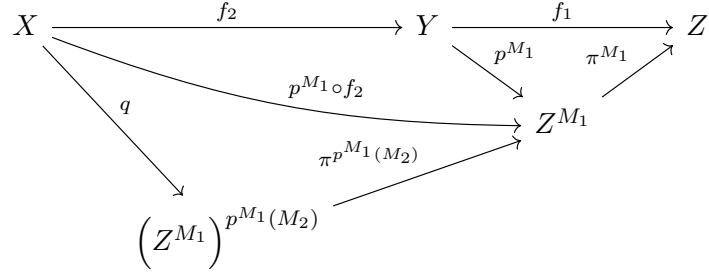

In this section, given an F-decomposable system, we construct an adjacent system, essentially preserving at every level only a specified subset of fibers. There are no requirements on these subsets, except for the natural condition that preserved fibers should correspond to existing points, that is, the projections of these points should lie in the preserved fibers on every lower level. In Section \ref{sect_proof} we will use this construction with the sets of preserved fibers corresponding to large oscillations of continuous functions. More precisely, we have the following:

Let $\{X_\alpha, \pi_\alpha^\beta, \alpha, \beta \in \gamma\}$ be an F-decomposable system. Let $M_\alpha \subset X_\alpha$ be subsets satisfying $\pi^\alpha_\beta (M_\alpha) \subset M_\beta$ for any $\beta \leq \alpha \in \gamma$. We will inductively construct an inverse system $\{K_\alpha, \lambda^\alpha_\beta, \alpha, \beta \in \gamma\}$ in the following way:

Let $K_0 := X_0$ and $K_1 := X_1$. Using the notation in \ref{quot}, we set $K_2 := (X_1)^{M_1}_{\pi_1^2}$. Having constructed the spaces $K_\beta$ and the corresponding mappings $q_\beta: X_\beta \to K_\beta$ for all $\beta \leq \alpha$ for some ordinal $\alpha \in \gamma$, we set the notation $q^\alpha_\beta := q_\beta \circ \pi^\alpha_\beta$ and $N_\alpha := q_\alpha(M_\alpha)$. Now, with $q^{\alpha + 1}_\alpha := q_\alpha \circ \pi^{\alpha+1}_\alpha$, we define
\begin{align*}
	K_{\alpha + 1} &:= (K_\alpha)^{N_\alpha}_{q^{\alpha + 1}_\alpha};\\
	\lambda_\alpha^{\alpha+1} &:= \left(q^{\alpha+1}_\alpha\right)^{N_\alpha}.
\end{align*}
Then $q_{\alpha + 1}$ is the unique mapping giving $q^{\alpha + 1}_\alpha = \lambda_\alpha^{\alpha + 1} \circ q_{\alpha + 1}$. For $\alpha$ a limit ordinal we set $K_\alpha := \lim\limits_{\leftarrow} \{K_\beta, \lambda^\delta_\beta, \beta, \delta \in \alpha\}$. The mapping $q_\alpha$ in this case is given as follows. For $x = \{x_\beta\}_{\beta \in \alpha} \in X$, we have $q_\alpha(x) = \{q_\beta(x_\beta)\}_{\beta \in \alpha}$. Finally, we define $K := \lim\limits_{\leftarrow} \{K_\alpha, \lambda_\alpha^\beta, \alpha, \beta \in \gamma\}$ and $q := q_\gamma$.

The mappings of the construction are summarized in the following diagram:

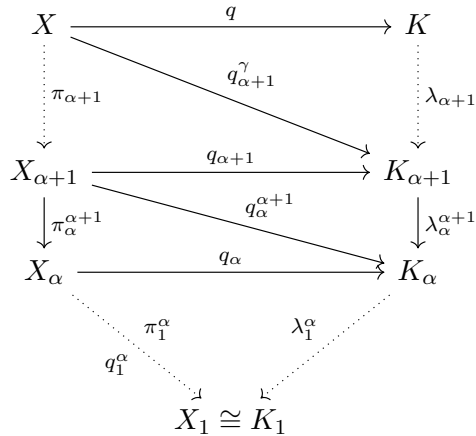
\begin{center}
	\begin{tikzcd}
		X \arrow[rr, "q"] \arrow[dd, dotted, "\pi_{\alpha+1}"]
		\arrow[ddrr, "q^\gamma_{\alpha+1}"] &
		&K \arrow[dd, dotted, "\lambda_{\alpha+1}"] \\
		\\
		X_{\alpha+1} \arrow[rr, "q_{\alpha+1}"] \arrow[d, "\pi^{\alpha+1}_\alpha"]
		\arrow[drr, "q^{\alpha+1}_\alpha"] &
		& K_{\alpha+1} \arrow[d, "\lambda^{\alpha+1}_\alpha"] \\
		X_\alpha \arrow[rr, "q_\alpha"] \arrow[ddr, dotted, "\pi^\alpha_1", "q^\alpha_1"'] &
		&K_\alpha \arrow[ddl, dotted,  "\lambda^\alpha_1"'] \\
		\\
		&X_1 \cong K_1 &
	\end{tikzcd}
\end{center}

\begin{proposition}\label{constr_well-def}
	For any $\beta \leq \alpha < \gamma$ the mapping $q_\beta^\alpha$ is fully closed. In particular the mappings $q^{\alpha + 1}_\alpha$ are fully closed and thus the spaces $K_\alpha$ are Hausdorff. Consequently, the system $\{K_\alpha, \lambda_\alpha^\beta, \alpha, \beta \in \gamma\}$ is F-decomposable.
\end{proposition}

\begin{proof}
	Suppose that we have shown that for some $\alpha \in \gamma$ all the mappings $q^\eta_\delta$ with $\delta \leq \eta < \alpha$ are fully closed. Observe that the mapping $q^\alpha_1 = \pi^\alpha_1$ is fully closed by supposition. We now proceed by induction on $\beta$ and suppose further that for some $\beta \leq \alpha$ the mappings $q^\alpha_\delta$ are fully closed for all $\delta < \beta$. First, if $\beta = (\beta - 1) + 1$, we can directly apply Proposition \ref{fed_level} with $X = X_\alpha$, $Y = X_\beta$, $Z = K_{\beta - 1}$, $M = N_{\beta - 1}$, $f = \pi^\alpha_\beta$ and $g = q^\beta_{\beta - 1}$, so $q^\alpha_\beta = q_\beta \circ \pi^\alpha_\beta$ is fully closed.
	
	Now consider the case of $\beta$ a limit ordinal. We shall first deal with the easier case where $\beta = \alpha$. To this end let $y \in K_\beta$. If $\lambda^\beta_\delta(y) \notin N_\delta$ for some $\delta < \beta$, we have that $\lambda^\beta_\delta$ is one-to-one at $y$. As $q^\beta_\delta$ is fully closed by the induction hypothesis and $q^\beta_\delta = \lambda^\beta_\delta \circ q_\beta$, Proposition \ref{comp_fc} gives us the full closedness of $q_\beta$ at $y$. If $\lambda^\beta_\delta(y) \in N_\delta$ for all $\delta < \beta$, then $q_\beta$ is one-to-one at $y$ and the conclusion follows.
	
	The remaining case is that of $\beta$ a limit ordinal with $\beta < \alpha$. Let  $y \in K_\beta$ and consider again the two cases. If $\lambda^\beta_\delta(y) \notin N_\delta$ for some $\delta < \beta$, the same argument as above gives us the full closedness of $q^\alpha_\beta$ at $y$.
	
	If $\lambda^\beta_\delta(y) \in N_\delta$ for all $\delta < \beta$, then $q_\beta$ is one-to-one at $y$. Let $U_1, \dots U_n$ be a finite open cover of the fiber $(q^\alpha_\beta)^{-1}(y) = (\pi^\alpha_\beta)^{-1}\left(q_\beta^{-1}(y)\right)$. From the full closedness of $\pi^\alpha_\beta$ we have that the set
	\begin{equation*}
		U := \bigcup_{i = 1}^n (\pi^\alpha_\beta)^\# U_i \cup \{q_\beta^{-1}(y)\}
	\end{equation*} 
	is a neighborhood of $q_\beta^{-1}(y)$.
	
	Denote $V = q_\beta^\#(U)$. As the last set is a neighborhood of $y$, there is a basic neighborhood $B \subset V$ of the form $B = \left(\lambda^\beta_{\delta_0}\right)^{-1}(B_{\delta_0})$ for some $\delta_0 < \beta$. Now assume that $U_1, \dots, U_n$ is not a cover of $(q^\alpha_{\delta_0})^{-1}\left(\lambda^\beta_{\delta_0}(y)\right)$. Then there exist $y' \in (\lambda^\beta_{\delta_0})^{-1}\left(\lambda^\beta_{\delta_0}(y)\right)$ and $x' \in (q^\alpha_\beta)^{-1}(y')$ such that $x' \notin U_1, \dots, U_n$. Then $y' \notin (q^\alpha_\beta)^\#\left(\bigcup\limits_{i = 1}^n U_i\right)$, but $B \setminus \{y\}$ is contained in the last set, a contradiction.
	
	Now for any $\delta \in \left[\delta_0, \beta\right)$, using the full closedness of $q^\alpha_\delta$, we have that $\bigcup\limits_{i = 1}^n (q^\alpha_\delta)^\# (U_i) \cup \{\lambda^\beta_\delta(y)\}$ is a neighborhood of $\lambda^\beta_\delta (y)$. Then $\left(\lambda^\beta_\delta\right)^{-1}\left(\bigcup\limits_{i = 1}^n (q^\alpha_\delta)^\# (U_i) \cup \{\lambda^\beta_\delta(y)\}\right)$ is a neighborhood of $y$. On the other hand
	\begin{align*}
		&\left(\lambda^\beta_\delta\right)^{-1}\left(\bigcup\limits_{i = 1}^n (q^\alpha_\delta)^\# (U_i) \cup \{\lambda^\beta_\delta(y)\}\right) \\
		= \quad & \bigcup\limits_{i = 1}^n \left(\lambda^\beta_\delta\right)^{-1} \left((q^\alpha_\delta)^\# (U_i)\right) \cup \left(\lambda^\beta_\delta\right)^{-1} \left(\lambda^\beta_\delta(y)\right) \\
		\subset \quad & \bigcup\limits_{i = 1}^n (q^\alpha_\beta)^\# (U_i) \cup \left(\lambda^\beta_\delta\right)^{-1} \left(\lambda^\beta_\delta(y)\right).
	\end{align*}
	We then obtain:
	\begin{align*}
		&\bigcup\limits_{i = 1}^n (q^\alpha_\beta)^\# (U_i) \\
		\supset \quad &\bigcup_{\delta_0 \leq \delta < \beta} \left(\bigcup\limits_{i = 1}^n \left(\lambda^\beta_\delta\right)^{-1} \left((q^\alpha_\delta)^\# (U_i)\right)\right) \\
		\supset \quad &\bigcup_{\delta_0 \leq \delta < \beta} \left(\left(\lambda^\beta_\delta\right)^{-1}\left(\bigcup\limits_{i = 1}^n (q^\alpha_\delta)^\# (U_i) \cup \{\lambda^\beta_\delta(y)\}\right) \setminus \left(\lambda^\beta_\delta\right)^{-1} \left(\lambda^\beta_\delta(y)\right) \right) \\
		= \quad &\bigcup_{\delta_0 \leq \delta < \beta} \left(\left(\lambda^\beta_\delta\right)^{-1}\left(\bigcup\limits_{i = 1}^n (q^\alpha_\delta)^\# (U_i) \cup \{\lambda^\beta_\delta(y)\}\right)\right) \; \setminus \; \bigcap_{\delta_0 \leq \delta < \beta} \left(\lambda^\beta_\delta\right)^{-1} \left(\lambda^\beta_\delta(y)\right).
	\end{align*}
	As $\bigcap\limits_{\delta_0 \leq \delta < \beta} \left(\lambda^\beta_\delta\right)^{-1} \left(\lambda^\beta_\delta(y)\right) = \{y\}$, $\bigcup\limits_{i = 1}^n (q^\alpha_\beta)^\# (U_i) \cup \{y\}$ is a neighborhood of $y$, which concludes the proof that $q^\alpha_\beta$ is fully closed. The statement now follows by induction on $\alpha$.
\end{proof}


\section{Continuous functions on limits of inverse systems} \label{characterization}

Recall that a tree is a partially ordered set $(\Upsilon, \preceq)$ such that for any $t_0 \in \Upsilon$ the subset $\{t \in \Upsilon: t \preceq t_0\}$ is well-ordered by $\preceq$. Following \cite{haydon_trees}, we append an element $0 \in \Upsilon$ that satisfies $0 \preceq t$ for all $t \in \Upsilon$. We consider on $\Upsilon$ the coarsest topology for which the intervals of the form $(0,t]$ are clopen. Observe that this topology is locally compact. We are only interested in trees that are Hausdorff. 

By $C_0(\Upsilon)$ we denote the Banach space of continuous functions on $\Upsilon$ for which $\{t \in \Upsilon: \abs{f(t)} \geq \epsilon\}$ is compact for any $\epsilon >0$. Equivalently, $C_0(\Upsilon)$ can be considered as the space of continuous functions on the one-point compactification $\alpha \Upsilon$ that vanish at $\infty$. The following simple observation is of great significance for the properties of this space (we refer the reader to \cite{haydon_trees} for more information on the topic).

\begin{proposition}[see e.g. {\cite[Lemma 2.1]{haydon_trees}}] \label{finite_antichanains}
	Let $\Upsilon$ be a tree, $f \in C_0(\Upsilon)$ and $\epsilon >0$. Then the set $\{t \in \Upsilon: \abs{f(t)} > \epsilon\}$ does not contain an infinite antichain.
\end{proposition}

\begin{remark}
	In fact, as the neighborhoods of $\infty$ in $\alpha \Upsilon$ are of the form $\Upsilon \setminus \bigcup_{i = 1}^k [0, t_i]$, we easily see that the conditions that $\{t \in \Upsilon: \abs{f(t)} > \epsilon\}$ does not contain an infinite antichain and that $f$ goes to $0$ on every unbounded increasing sequence are equivalent to the continuity of $f$ at $\infty$.
\end{remark}

Now let $S:= \{X_\alpha, \pi_\alpha^\beta, \alpha, \beta \in \gamma\}$ be a continuous inverse system of Hausdorff compacta, where $\gamma$ is some ordinal and let $X := \lim\limits_{\leftarrow} S$. Define $\Upsilon(S) := \bigcup \{X_\alpha, \alpha \in \gamma\}$ and a partial order $\preceq$ on $\Upsilon(S)$ as follows. If $y \in X_\alpha, x \in X_\beta$, then
\begin{equation*}
	y \preceq x \quad \iff \quad \alpha \leq \beta \text{ and } \pi_\alpha^\beta(x) = y.
\end{equation*}
We shall refer to $\Upsilon(S)$ as the \textit{tree of the system} $S$. It is straightforward to check that $(\Upsilon(S), \preceq)$ is indeed a tree. We now consider the following mapping:
\begin{align*}
	\Phi: C(X) \to \l_\infty(\Upsilon(S)) \\
	\Phi(f)(x) := \osc_{\pi_\alpha^{-1}(x)} f,
\end{align*}
whenever $x \in X_\alpha$, $\alpha \in \gamma$.

Before we proceed with the main result of this section, let us introduce some notation, inspired by Section \ref{constr}, that will be useful for what follows as well as in Section \ref{sect_proof}.

Let $S:= \{X_\alpha, \pi_\alpha^\beta, \alpha, \beta \in \gamma\}$ be a continuous inverse system of Hausdorff compacta, where $\gamma$ is some ordinal and let $X := \lim\limits_{\leftarrow} S$. For a function $f \in C(X)$, $\epsilon >0$, and $\alpha \in \gamma$ we set $M^{f, \epsilon}_\alpha := \{x \in X_\alpha: \osc\limits_{\pi_\alpha^{-1}(x)} f \geq \epsilon\}$. Further, we set $\nu_f := \{\alpha \in \gamma: \pi^{\alpha + 1}_\alpha(M^{f, \epsilon}_{\alpha+1}) \neq M^{f, \epsilon}_\alpha\}$. Observe that if $S$ is F-decomposable, the set $M^{f, \epsilon}_\alpha$ is finite for any $\alpha \in \mu$ by Proposition \ref{fc_limit} and Proposition \ref{c_0.osc}. Moreover, we have the following:

\begin{lemma} \label{finite_branching}
	Let $S:= \{X_\alpha, \pi_\alpha^\beta, \alpha, \beta \in \gamma\}$ be an F-decomposable system and $X$ be its limit. Then for any $f \in C(X)$,  $\nu_f$ is finite.
\end{lemma}

\begin{proof}
	Let $f \in C(X)$ and assume that the set $\nu_f$ is infinite. Then it has an accumulation point, denote it by $\alpha_0$. Observe that if $\gamma$ is a limit ordinal, then the sets of the form $\pi_\alpha^{-1}(U)$ for $\alpha \in \gamma$ and $U$ open in $X_\alpha$ form a base for the topology of $X$. We can thus take an open cover $\{U^{f,\epsilon}_x: x \in X\}$ consisting of basic neighborhoods such that $\osc_{U^{f,\epsilon}_x} f < \epsilon$. Then from the compactness of $X$ we find a finite subcover $\{U_1, \dots U_n\}$. As each $U_i = \pi_{\alpha_i}^{-1}(V_{\alpha_i})$ for some $V_{\alpha_i}$ open in $X_{\alpha_i}$, we have that $\osc_{\pi_\beta^{-1}(x)} f < \epsilon$ for all $\beta \geq \max \{\alpha_1, \dots, \alpha_n\}$. This shows that $\alpha_0 \neq \gamma$.
	
	Now suppose that $\max\{\osc_{\pi_{\alpha_0}^{-1}(x)} f: x \in X_{\alpha_0}\} = \kappa < \epsilon$ for all $x \in X_{\alpha_0}$. Then by Proposition \ref{metr_proj} there exists a function $g \in C(X_{\alpha_0})$ such that $\|g -f\| = \frac{\kappa}{2}$. By the same argument as above, we find $\beta_0 < \alpha_0$ such that $\osc_{\pi_{\beta_0}^{-1}(x)} g < \epsilon - \kappa$ for all $\beta \in [\beta_0, \alpha_0)$. We then have
	\begin{equation*}
		\osc_{\pi_{\beta_0}^{-1}(x)} f  \leq \osc_{\pi_{\beta_0}^{-1}(x)} g + 2\|f-g\|< \epsilon - \kappa + 2 \frac{\kappa}{2} = \epsilon.
	\end{equation*}
	Thus $M^{f, \epsilon}_\beta = \emptyset$ for all $\beta \in [\beta_0, \alpha_0)$, which contradicts the assumption that $\alpha_0$ is an accumulation point of $\nu_f$.
	
	Suppose now that $N_{\alpha_0} := M^{f, \epsilon}_{\alpha_0}$ is nonempty. Define $N_\beta := \pi^{\alpha_0}_\beta (N_{\alpha_0})$ for $\beta < \alpha_0$ and $N_\beta := (\pi^\beta_{\alpha_0})^{-1} (N_{\alpha_0})$ for $\beta > \alpha_0$. We will apply the construction of Section \ref{constr} with the sets $\{N_\alpha: \alpha \in \gamma\}$. Denote the resulting inverse system $\{K_\alpha, \lambda^\alpha_\beta, \gamma\}$, its limit $K$ and $q_\alpha: X_\alpha \to K_\alpha$ the corresponding mappings, as defined in Section \ref{constr}. By construction $\lambda_{\alpha_0}^{-1}(q_{\alpha_0}(N_{\alpha_0}))$ and $\pi_{\alpha_0}^{-1}(N_{\alpha_0})$ are homeomorphic. We can thus regard the restriction of $f$ to $\pi_{\alpha_0}^{-1}(N_{\alpha_0})$ as a continuous function on $\lambda_{\alpha_0}^{-1}(q_{\alpha_0}(N_{\alpha_0}))$. For simplicity we will again denote this function by $f$. Now as the set $N_{\alpha_0}$ is finite, $\lambda_{\alpha_0}^{-1}(q_{\alpha_0}(N_{\alpha_0}))$ is closed in the Hausdorff compact space $K$. We use the Tietze  theorem to obtain a continuous extension $\tilde{f} \in C(K)$ of $f$. Once again preserving notation after embedding, we obtain the function $\tilde{f} \in C(X)$ that satisfies $\osc_{\pi_\beta^{-1}(x)} \tilde{f} = 0$ for all $\beta \in \gamma$ and all $x \notin N_\beta$.
	
	Consider the function $g := f - \tilde{f}$. We have that $g\vert_{\pi_{\alpha_0}^{-1}(N_{\alpha_0})} = 0$ and $\osc_{\pi_{\alpha_0}^{-1}(x)} g = \osc_{\pi_{\alpha_0}^{-1}(x)} f$ for all $x \in X_{\alpha_0} \setminus N_{\alpha_0}$. Consequently, $\osc_{\pi_{\alpha_0}^{-1}(x)} g < \epsilon$ for all $x \in X_{\alpha_0}$. As in the first case we considered for the function $f$, we find an ordinal $\beta_0 < \alpha_0$ such that $\osc_{\pi_\beta^{-1}(x)} g < \epsilon$ for all $\beta \in [\beta_0, \alpha_0)$ and $x \in X_\beta$.
	
	We directly obtain $\osc_{\pi_\beta^{-1}(x)} f < \epsilon$ for all $\beta \in [\beta_0, \alpha_0)$ and all $x \in X_\beta \setminus N_\beta$. Thus $M^{f, \epsilon}_\beta \subset N_\beta$ for all $\beta \in [\beta_0, \alpha_0)$ and consequently $\nu_f \cap [\beta_0, \alpha_0) = \{\beta \in [\beta_0, \alpha_0): \osc_{\pi_\beta^{-1}(\pi^{\alpha_0}_\beta(x))} f \geq \epsilon \text{ and } \osc_{\pi_{\beta+1}^{-1}(\pi^{\alpha_0}_{\beta+1}(x))} f < \epsilon \text{ for some } x \in N_{\alpha_0}\}$. As $N_{\alpha_0}$ is finite, the last set also has to be finite, thus contradicting the assumption that $\alpha_0$ is an accumulation point of $\nu_f$. 
\end{proof}

\begin{corollary} \label{charact_tree}
	Let $S := \{X_\alpha, \pi_\alpha^\beta, \alpha, \beta \in \gamma\}$ be a continuous inverse system, $X$ its limit and $\Upsilon(S)$ the tree of $S$. Then $\Phi$ maps $C(X)$ into $C_0(\Upsilon(S))$ if and only if $S$ is F-decomposable.
\end{corollary}

\begin{proof}
	Necessity follows directly from Proposition \ref{c_0.osc} and Proposition \ref{finite_antichanains}. Indeed, if for some $\alpha \in \gamma$ $\pi_\alpha$ is not fully closed, then there is a function $f \in C(X)$ and $\epsilon >0$ such that $M^{f, \epsilon}_\alpha$ is infinite. However, the points of this set form an antichain in $\Upsilon(S)$ on which $\Phi(f) > \epsilon$.
	
	Now let $f \in C(X)$ and $x \in X_{\alpha_0}$ for some limit ordinal $\alpha_0 \in \gamma$. As $\pi_{\alpha_0}^{-1} (x) = \bigcap_{\alpha \in \alpha_0} \pi_\alpha^{-1}(\pi^{\alpha_0}_\alpha(x))$, we have that $\lim \Phi(f)(x_\alpha) = \Phi(f)(x)$, whenever $\lim x_\alpha = x$. Assume that $S$ is F-decomposable. Having in mind the remark after Proposition \ref{finite_antichanains}, we need to show two things. First, if $(t_\alpha)_{\alpha \in \gamma}$ is an unbounded increasing sequence (this case is relevant only for $\gamma$ a limit ordinal), $f(t_\alpha) \to 0$ by the proof of Lemma \ref{finite_branching}. Now suppose that there is an infinite antichain $A \in \Upsilon$ such that $\Phi(f)(t) \geq \epsilon$ for all $t \in A$. For every $t \in A$ take a maximal element of the set $\{x \in [t, \infty): \Phi(f)(x) \geq \epsilon\}$. Denote this element $x_t$ and its level $\alpha_t$, i.e. $x_t \in X_{\alpha_t}$. Then $x_t \in M^{f, \epsilon}_{\alpha_t}$ and from maximality we have that $(\pi^{\alpha_t+1}_{\alpha_t})^{-1}(x_t) \cap M^{f, \epsilon}_{\alpha_t + 1} = \emptyset$. Consequently, $\alpha_t \in \nu_f$. But the last set is finite by Lemma \ref{finite_branching}. Note that because $A$ is an antichain, we necessarily have $x_t \neq x_s$ whenever $s$ and $t$ are two distinct elements of $A$. As $A$ is infinite, we can find an ordinal $\alpha_0 \in \gamma$ such that $\alpha_t = \alpha_0$ for every $t$ in some infinite $B \subset A$. However, $\{x_t: t \in B\} \subset M^{f, \epsilon}_{\alpha_0}$ which contradicts the full closedness of $\pi_{\alpha_0}$ and thus the F-decomposability of the system.
\end{proof}

\section{$\mathrm{F}_d$ compacta and LUR renormability} \label{sect_proof}

This section is devoted to the proof of our main result, which is the following:

\begin{theorem}\label{main_thm}
	Let $X = \lim\limits_{\leftarrow} \{X_\alpha, \pi_\alpha^\beta, \alpha, \beta \in \gamma\}$ be an $\mathrm{F}_d$-compact of spectral height $\gamma < \omega_1$. Then $C(X)$ admits an equivalent $\tau_p$-lsc LUR norm.
\end{theorem} 

\begin{proof}
	Let $\epsilon > 0$. We will find a countable decomposition of the space $C(X)$ that satisfies the conditions of Theorem \ref{LUR_char}. We shall proceed in several steps.
	
	As for any $f \in C(K)$ we have $\pi^\alpha_\beta\left(M^{f,\epsilon}_\alpha\right) \subset M^{f,\epsilon}_\beta$, we can once again apply the construction from Section \ref{constr}. Denote the resulting compact space $K^{f, \epsilon}$. The full closedness of the bonding mappings implies that the sets $M^{f,\epsilon}_\alpha$ are finite for all $\alpha \in \gamma$. Thus the space $K^{f,\epsilon}$ is metrizable. The above assertion allows us to start with the first decomposition as follows:
	\begin{equation} \label{dec1}
		C(X) = \bigcup_{\nu \in \gamma^{<\omega}} E_\nu: \quad f \in E_\nu, \text{ if } \nu_f = \nu.
	\end{equation}
	
	In what follows we shall denote by $\lgth(\nu)$ the length of a finite sequence. For the subsequent decomposition we make use of the fact that each $M^{f, \epsilon}_{\nu_i}$ is finite:
	\begin{equation} \label{dec2}
		E_\nu = \bigcup_{k \in \N^{\lgth(\nu)}} E_{\nu,k}: \quad f \in E_{\nu,k}, \text{ if } \abs{M^{f,\epsilon}_{\nu_i}} = k_i; \quad i = 1, \dots, \lgth(\nu). 
	\end{equation}
	
	Often the construction of a countable decomposition satisfying the hypotheses of Theorem \ref{LUR_char} requires passing from a weak (pointwise) neighborhood of a given point to an open (pointwise open) halfspace. In this note we will do this with a technique similar to the one in \cite[Lemma 4.22]{motv}. In particular, we have the following:
	\begin{equation}
		\begin{aligned}\label{dec3}
			&E_{\nu, k} = \bigcup_{r \in \N} E_{\nu, k, r}: \quad f \in E_{\nu, k, r}, \text{ if }\\
			&\epsilon - \max\left\{\osc_{\pi_{\nu_i}^{-1}(x)} f: \;  i \in \{1, \dots, \lgth(\nu)\}, x \in \left(X_{\nu_i} \setminus M^{f, \epsilon}_{\nu_i}\right)\right\}> \frac{1}{r}.
		\end{aligned}
	\end{equation}
	
	Note that the maximum above is well defined by the full closedness of $\pi_{\nu_i}$. The next decomposition will be used bellow together with (\ref{dec3}) to combine a finite number of functionals into one. Essentially, it allows us to choose the functionals associated with a function $f \in C(X)$ in such a way that they almost attain their supremum on the corresponding piece of the decomposition at $f$. We will need to well-order the spaces $X_\alpha$ to ensure that what follows is well defined.
	\begin{equation}\label{dec4}
		E_{\nu, k, r} = \bigcup_{q \in \Q^{\sum k_i}} E_{\nu, k, r, q}: \quad \osc_{\pi_{\nu_i}^{-1}(x_j)} f \in \left(q_j - \frac{1}{3 r \sum_i k_i}, q_j + \frac{1}{3 r \sum_i k_i}\right),
	\end{equation}
	where $x_j$ are the points of $M^{f, \epsilon}_{\nu_i}$, ordered first by the index of their space $\nu_i$ and then by the well-ordering of $X_{\nu_i}$. The index $q$ is thus a multi index with $\sum_{i = 1}^{\lgth(\nu)} k_i$ coordinates.
	
	For the last decomposition we will use the construction of Section \ref{constr} with the sets $M^{f,\epsilon}_\alpha$. Denote the resulting compact space $K_f^\epsilon$. We can inductively show that this space is metrizable. Indeed, if we suppose some lower level space $K_{f, \alpha}^\epsilon$ is metrizable, the same will be true for $K_{f, \alpha +1}^\epsilon$ by Proposition \ref{Fed_metr} and the fact that the set of nontrivial fibers is finite. Now if $\alpha \in \gamma$ is a limit ordinal and $K_{f, \beta}^\epsilon$ are metrizable for all $\beta \in \alpha$, $K_{f, \alpha}^\epsilon$ will also be metrizable as it is homeomorphic to a subspace of a countable product of metrizable spaces. We thus have that $C(K_f^\epsilon)$ is separable and we can fix a countable dense subset $\{h_p\}_{p \in \N} \subset C(K_f^\epsilon)$. Then we set:
	\begin{equation}\label{dec5}
		E_{\nu, k, r, q} = \bigcup_{p \in \N} E_{\nu, k, r, q, p}: \quad  f \in E_{\nu, k, r, q,p}, \text{ if } f \in B_{\epsilon}(h_p).
	\end{equation}
	This is possible because $q^\epsilon_f: X \to K^\epsilon_f$ is a continuous surjection and by Proposition \ref{metr_proj} we have $\dist(f, C(K^\epsilon_f)) = \frac{1}{2} \sup\{\osc_{(q^\epsilon_f)^{-1}(x)}: x \in K^\epsilon_f\} < \epsilon$.
	
	We are now ready to select the functional giving the required halfspace. For $f \in E_{\nu,k,r,q,p}$ and $x \in M^{f, \epsilon}_{\nu_i}$ for some $i \leq \lgth(\nu)$, we choose $\overline{x}, \underline{x} \in \pi_{\nu_i}^{-1} (x)$ such that $f(\overline{x}) - f(\underline{x}) = \osc_{\pi_{\nu_i}^{-1}(x)} f$. We can then define the functional $\mu^f_x \in \left(C(X), \tau_p\right)^*$ as follows:
	\begin{equation*}
		\mu^f_x(g) := g(\overline{x}) - g(\underline{x}).
	\end{equation*}
	Finally, we define
	\begin{equation*}
		\mu^f := \sum \left\{\mu^f_x : \; x \in \bigcup\limits_{i = 1}^{\lgth(\nu)} M^{f, \epsilon}_{\nu_i} \right\}.
	\end{equation*}
	The required halfspace will then be given by:
	\begin{equation}
		H_f := \left\{g \in C(X): \; \mu^f (g) > \mu^f (f) - \frac{1}{3r}\right\}.
	\end{equation}
	
	\begin{claim}
		If $g \in E_{\nu,k,r,q,p} \cap H_f$, then $\mu^f_x(g) > \epsilon$ for all $x \in \bigcup_{i = 1}^{\lgth(\nu)} M^{f, \epsilon}_{\nu_i} $.
	\end{claim}
	
	\begin{proof}
		Observe first that if $\mu^f_x (g) \geq \epsilon$, for some $i \leq \lgth(\nu)$ and $x \in X_{\nu_i}$, then $x \in M^{g,\epsilon}_{\nu_i}$. As $g \in E_{\nu,k,r,q,p}$, we have that $\osc\limits_{\pi_i^{-1}(x)} g \in \left(q_j - \frac{1}{3 r \sum_i k_i}, q_j + \frac{1}{3 r \sum_i k_i}\right)$ for some $j \in \{1, \dots, \sum k_i\}$. Also, for all $j \in \{1, \dots, \sum k_i\}$, the inequality $q_j > \epsilon  - \frac{1}{3 r \sum_i k_i}$ holds.
		
		Let $\{x_1, \dots, x_l\} \subset \bigcup\limits_{i = 1}^{\lgth(\nu)} M^{f,\epsilon}_{\nu_i}$ be the subset consisting of points that satisfy $\mu^f_{x} (g) < \epsilon$, $x \in \{x_1, \dots, x_l\}$. We have:
		\begin{align*}
			&\sum \left\{\mu^f_x (g): \; x \in \left(\bigcup_{i = 1}^{\lgth(\nu)} M^{f, \epsilon}_{\nu_i}\right) \setminus \{x_1, \dots, x_l\}\right\} \\
			< \quad &\sum_{j = 1}^{\sum k_i} q_j + \left(\sum_{i = 1}^{\lgth(\nu)} k_i\right) \frac{1}{3 r \sum_{i = 1}^{\lgth(\nu)} k_i} - l\left(\min_j q_j + \frac{1}{3 r \sum_{i = 1}^{\lgth(s)} k_i}\right) \\
			< \quad& \mu^f (f) + 2 \left(\sum_{i = 1}^{\lgth(\nu)} k_i\right) \frac{1}{3 r \sum_{i = 1}^{\lgth(\nu)} k_i} - l\epsilon \\
			= \quad&\mu^f (f) + \frac{2}{3r} - l \epsilon.
		\end{align*}
		On the other hand, as $g \in H_f$, we have
		\begin{align*}
			&\sum \left\{\mu^f_x (g): \; x \in \left(\bigcup_{i = 1}^{\lgth(\nu)} M^{f, \epsilon}_{s_i}\right) \setminus \{x_1, \dots, x_l\}\right\} \\
			&> \quad \mu^f (g) - l \left(\epsilon - \frac{1}{r}\right) \\
			&> \quad \mu^f (f) - \frac{1}{3r} - l \left(\epsilon - \frac{1}{r}\right) \\
			&= \quad \mu^f (f) - \frac{1 - 3l}{3r} - l\epsilon,
		\end{align*}
		a contradiction whenever $l \geq 1$.
	\end{proof}
	
	By Lemma \ref{finite_branching} and the definition of $\nu$ the sets $M^{f, \epsilon}_\alpha$ are completely determined by the sets $M^{f, \epsilon}_{\nu_i}$. It follows that if $g \in E_{\nu,k,r,q,p} \cap H_f$, then $M^{g, \epsilon}_\alpha = M^{f, \epsilon}_\alpha$ for all $\alpha \in \gamma$ and thus $K^\epsilon_g = K^\epsilon_f$. From the decomposition (\ref{dec5}) we have that $\|f - h_p\| < \epsilon$ and $\|g - h_p\| < \epsilon$. Thus $\|f-g\| < 2 \epsilon$, which concludes the proof.
\end{proof}

Faculty of Mahtematics and Informatics, Sofia University "St. Kliment Ohridski", 5, James Bourchier blvd., 1164 Sofia, Bulgaria

Email address: \href{mailto:tmmanev@fmi.uni-sofia.bg}{tmmanev@fmi.uni-sofia.bg} 


\begin{thebibliography}{[16]}
	
	
	
	
	\normalsize
	\baselineskip=17pt
	
	
	
	\bibitem{av_kosz} A. Avil\'es, P. Koszmider, \emph{A continuous image of a Radon-Nikod\'ym compact space which is not Radon-Nikod\'ym}, Duke Math. J., Vol. 162, no. 12, 2013, 2285--2299.
	
	\bibitem{dgz}R. Deville, G. Godefroy, and V. Zizler, \emph{Smoothness and renormings in Banach spaces}, Longman Scientific Technical, Harlow, 1993.
	
	\bibitem{eng}R. Engelking, \emph{General Topology}, Warszawa, PWN, 1985.
	
	\bibitem{fed69}V. V. Fedorchuk, \emph{Strongly closed mappings}, Sov. Math., Dokl., Vol. 10, 1969, 804--806.
	
	\bibitem{fedFC}V. V. Fedorchuk, \emph{Fully closed mappings and their applications}, J. Math. Sci., Vol. 136, 2006, 4201--4292.
	
	\bibitem{gist}S. P. Gul’ko, A. V. Ivanov, M. S. Shulikina, and S. Troyanski, \emph{Locally uniformly rotund renormings of the spaces of continuous functions on Fedorchuk compacts}, Topol. Appl., Vol. 281, 2020, 1--11.
	
	\bibitem{haydon_trees} R. Haydon, \emph{Trees in renorming theory}, Proc. Lond. Math. Soc., Vol. 78, no. 3, 1999, 541--584.
	
	\bibitem{hjnr00}R. Haydon, J. Jayne, I Namioka, and C. Rogers, \emph{Continuous functions on totally ordered spaces that are compact in their order topologies}, J. Func. Anal., Vol. 178, 2000, 23--63.
	
	\bibitem{hmo-count.disc} R. Haydon, A. Molt\`o, J. Orihuela, \emph{Spaces of functions with countably many discontinuities}, Israel J. Math., Vol. 158, 2007, 19--39.
	
	\bibitem{hay-rog_scattered}R. Haydon and C. Rogers, \emph{A locally uniformly convex renorming for certain C(K)}, Mathematika, Vol. 37, 1990, 1--8.
	
	\bibitem{holmes_opt} R. Holmes, \emph{A Course on Optimization and Best Approximation}, Lecture Notes in Math., Vol. 257, Springer-Verlag, 1972.
	
	\bibitem{ivanov1984fedorchuk}A. V. Ivanov, \emph{On Fedorchuk compacta}, in: Mappings and Functors, Izd. Mosk. Univ. Moscow, 1984, 31--40 (in Russian).
	
	\bibitem{ivanov_qfc} A. V. Ivanov, \emph{Almost fully closed mappings and quasi-F-compacta}, Trans. Karelian Res. Cent. RAS 7, 2018, 25--33 (in
	Russian).
	
	\bibitem{kosz-few_op} P. Koszmider, \emph{Banach spaces of continuous functions with few operators}, Math. Ann., Vol. 330, no. 1, 2004, 151--183.
	
	\bibitem{manev24} T. Manev, \emph{Fully closed mappings and LUR renormability}, Studia Math., Vol. 278, no. 1, 2024, 67--79.
	
	\bibitem{mirna_pleb} D. Mirna, G. Plebanek, \emph{On Efimov spaces and Radon measures}, Topol. Appl., Vol. 154, no. 10, 2007, 2063--2072.
	
	\bibitem{mot97} A. Molt\'o, J. Orihuela, and S. Troyanski, \emph{Locally uniformly rotund renorming and fragmentability}, Proc. London Math. Soc., Vol. 75, 1997, 619--640.
	
	\bibitem{motv} A. Molt\'o, J. Orihuela, S. Troyanski, and M. Valdivia, \emph{A nonlinear transfer technique for renorming}, Lecture Notes in Math., Vol. 1951, Springer, 2009.
	
	\bibitem{watson} S. Watson, \emph{The construction of topological spaces, planks and resolutions}, in: Recent progress in general topology, M. Hušek, J. van Mill (Eds.), Vol. 20, North-Holland, 1992, 673–757.
	
\end{thebibliography}
\end{document}